\documentclass{amsart}%
\usepackage{amsmath}
\usepackage{amssymb}
\usepackage{amsthm}
\usepackage{amscd}
\usepackage[dvipdfmx]{graphicx}
\usepackage[mathscr]{eucal}
\usepackage{bm}
\usepackage[all]{xy} 
\makeindex
\newtheorem{Thm}{Theorem}[section]
\theoremstyle{definition}
\newtheorem{Theorem}[Thm]{Theorem}
\newtheorem{Lemma}[Thm]{Lemma}

\theoremstyle{remark}
\newtheorem{Remark}{Remark}
\font\sy=cmsy10

\font\ym=msbm10

\newcommand{\C}{\text{\ym C}}

\newcommand{\cL}{{\hbox{\sy L}}}
\newcommand{\cM}{{\hbox{\sy M}}}

\newcommand{\sB}{\mathscr B}

\newcommand{\Hom}{\hbox{\rm Hom}}
\newcommand{\End}{\hbox{\rm End}}

\title[bicategory of bimodules]{Notes on the bicategory of W*-bimodules}
\author{Yusuke Sawada \and Shigeru Yamagami}
\begin{document}
\maketitle
\begin{center}
Graduate School of Mathematics
\end{center}
\begin{center}
Nagoya University 
\end{center}
\begin{center} 
Nagoya, 464-8602, JAPAN
\end{center}    


\begin{abstract}
Categories of W*-bimodules are shown in an explicit and algebraic way 
to constitute an involutive W*-bicategory. 
\end{abstract} 

\section*{Introduction}
As in the purely algebraic case, it is fairly obvious to extract a bicategory from 
operator-algebraic bimodules provided that the relevant monoidal structure is based on 
ordinary module tensor products. 

Although it needs an operator-algebraic modification 
to have natural tensor products (see \cite{conn}, \cite{sauv})  , 
we know that W*-bimodules (i.e., Hilbert spaces with von Neumann algebras acting continuously) 
still supply a bicategory of W*-algebraic nature, called a W*-bicategory (see \cite{yama2}, \cite{yama4}). 

In the present notes, we shall show that the whole construction of the bicategory in question 
as well as the accompanied involution is possible in an explicit and algebraic manner
without detailed knowledge of modular theory. 







The organization is as follows: 
Related with W*-bimodules, we introduce two W*-bicategories $\cM^\leftthreetimes$, $\cM^\rightthreetimes$  
and show that these are monoidally equivalent based on the unit object property of 
standard W*-bimodules. 

We then notice the fact that the operation of taking dual bimodules gives an anti-multiplicative 
equivalence between 
$\cM^\leftthreetimes$ and $\cM^\rightthreetimes$, 
which is utilized to get involutions on $\cM^\leftthreetimes$ and $\cM^\rightthreetimes$ respectively so that 
the monoidal equivalence between these preserves involutions as well. 

In this way, we have a single W*-bicategory with involution, 
which recovers the one dealt with in \cite{yama2}.

\section{Preliminaries}
\subsection{Bicategories}
In this paper, a linear bicategory is simply referred to as a bicategory 
(see \cite{macl} for categorical backgrounds). 
Thus our bicategory is a kind of categorification of linear algebra and consists of a family 
of linear categories ${}_A\cL_B$ indexed by a pair $(A,B)$ of labels with the following information: 
\begin{itemize}
\item
A special object $I_A$ (called a unit object) in ${}_A\cL_A$ is assigned to each label $A$. 
\item
A bivariant functor $(\cdot)\otimes (\cdot):{}_A\cL_B\times {}_B\cL_C \to {}_A\cL_C$ is assigned to
each triplet $(A,B,C)$ of labels. 
\item 
Isomorphisms (called unit isomosphisms) 
$l_X: I_A\otimes X \to X$ and $r_X:X\otimes I_B \to X$ are assigned to each object $X$ in ${}_A\cL_B$. 
\item 
An isomorphism 
(called associativity isomorphism) $a_{X,Y,Z}:(X\otimes Y)\otimes Z \to X\otimes (Y\otimes Z)$ is assigned to 
each triplet $(X,Y,Z)$ which is admissible in the the sense that 
$X \in {}_A\cL_B$, $Y \in {}_B\cL_C$ and $Z \in {}_C\cL_D$ for some labels $A,B,C,D$. 
\end{itemize}
These are then required to satisfy the following conditions: 
\begin{enumerate}
\item
$l_X$ and $r_X$ are natural in $X$ and satisfy the triangle identity 
in the sense that they make the following triangular diagrams commutative. 
\item 
$a_{X,Y,Z}$ is natural in $X$, $Y$, $Z$ and satisfies the pentagon identity in the sense that 
it makes the following pentagonal diagrams commutative. 
\end{enumerate}

\[
\begin{xy}
(0,28) *{(X\otimes I)\otimes Y}="A",
(30,28) *{X\otimes (I\otimes Y)}="B",
(15,8) *{X\otimes Y}="C",
(0,0) *{}="D",
\ar "A";"B"
\ar "A";"C"
\ar "B";"C"
\end{xy}
\hspace{-6mm}
\begin{xy}
(23,30) *{(W\otimes X)\otimes (Y\otimes Z)}="A", 
(0,15) *{((W\otimes X)\otimes Y)\otimes Z}="B", 
(46,15) *{W\otimes (X\otimes (Y\otimes Z))}="C", 
(0,0) *{(W\otimes (X\otimes Y))\otimes Z}="D",
(46,0) *{W\otimes ((X\otimes Y)\otimes Z)}="E", 
\ar "B";"A"
\ar "A";"C"
\ar "B";"D"
\ar "D";"E"
\ar "E";"C"
\end{xy}
\]

If a bicategory consists of C*-categories (or W*-categories) 
and all the relevant morphisms are unitary, it is called 
a C*-bicategory (or a W*-bicategory). See \cite{glr} (cf.~also \cite{yama4}) 
for more information on operator categories. 

By an {\bf involution} on a bicategory $\cL$, 
we shall mean a family of contravariant 
functors ${}_A\cL_B \to {}_B\cL_A$, 
which we denote by 
\[ 
X \mapsto X^*, 
\quad 
\Hom(X,Y) \ni f \mapsto {}^tf \in \Hom(Y^*,X^*)
\]
for objects $X$, $Y$ in ${}_A\cL_B$,  
together with natural families of isomorphisms 
$\{ c_{X,Y}: Y^*\otimes X^* \to (X\otimes Y)^*\}$ (anti-multiplicativity) 
and 
$\{ d_X: X \to (X^*)^*\}$ (duality) 
making the following diagrams commutative 
\[
\begin{CD}
(X^*\otimes Y^*)\otimes Z^* @>{c\otimes 1}>> 
(Y\otimes X)^*\otimes Z^* @>{c}>> 
(Z\otimes(Y\otimes X))^*\\
@V{a}VV & & @VV{{}^ta}V\\
X^*\otimes(Y^*\otimes Z^*) @>>{1\otimes c}> 
X^*\otimes(Z\otimes Y)^* @>>{c}> 
((Z\otimes Y)\otimes X)^*
\end{CD}\ ,
\]
\[
\begin{CD}
X\otimes Y @>{d\otimes d}>> X^{**}\otimes Y^{**}\\
@V{d}VV @VV{c}V\\
(X\otimes Y)^{**} @>>{{}^tc}> (Y^*\otimes X^*)^*
\end{CD}\ ,
\]
and fulfilling ${}^td_X = d^{-1}_{X^*}: X^{***} \to X^*$.
(The naturality means ${}^t(f\otimes g) \stackrel{c}\sim {}^tg\otimes{}^tf$ and  
$f \stackrel{d}\sim {}^t({}^tf)$.) 

We remark here that the operation $(X \mapsto X^*, f \mapsto {}^tf)$ 
together with $c$ satisfying anti-multiplicativity is
an anti-monoidal functor and we see that $f \mapsto {}^t({}^tf)$ gives 
a monoidal functor with the multiplicativity 
\[
{}^tc^{-1}\circ c: X^{**}\otimes Y^{**} \to (Y^*\otimes X^*)^* 
\to (X\otimes Y)^{**}.
\]

In literature, our involution is named in various ways; it is referred to as, for example, 
having duals in \cite{bw} with extra conditions assumed in connection with unit objects, 
which turns out to be redundant. 

For C*-bicategories (especially for W*-bicategories), it is natural to
assume the compatibility with the *-operation on morphisms as studied in \cite{yama3}; 
all the relevant structural isomorphisms 
are assumed to be unitary and the operation $f \mapsto {}^tf$ satisfies ${}^t(f^*) = ({}^tf)^*$. 
To avoid the confusion in this situation, we have used the different symbols $X^*$ and ${}^tf$ to
denote a single functor. (Other remedy is to use ${}^tX$ for objects, 
which looks however apparently awkward.)

\subsection{Bimodules}
We here review relevant facts from \cite{yama2} (cf.~\cite{bdh} also). 
Let $A$ and $B$ be W*-algebras. 
By an ${}_A\text{W*}_B$ bimodule $X$, we shall mean a Hilbert space $X$ 
on which W*-algebras $A$ and $B$ are are normally (i.e., weak-continuously) represented 
in an $A$-$B$ bimodule fashion. 
We often write ${}_AX_B$ to indicate the acting algebras. 
Given ${}_A\text{W*}_B$ bimodules $X$ and $Y$, the Banach space of bounded 
$A$-$B$ linear maps of ${}_AX_B$ into ${}_AY_B$ is denoted by $\Hom(X,Y)$. 
With these as hom-sets, we have the W*-category of ${}_A\text{W*}_B$ bimodules, which 
is denoted by ${}_A\cM_B$ in what follows. We regard a left W*-$A$ module (resp.~a right W*-$B$ module) as 
an ${}_A\text{W*}_B$ bimodule for $B=\C$ (resp.~for $A = \C$). 

The so-called standard representation (space) of a W*-algebra $A$ (\cite{haag1})
is nothing but the regular representation of $A$ and denoted by $L^2(A)$ in this paper. 
Recall that $L^2(A)$ is an ${}_A\text{W*}_A$ bimodule, 
which is linearly spanned by symbols $\varphi^{1/2}$ ($\varphi \in A_*^+$) so that 
the reduced left GNS space $\overline{[\varphi]A\varphi^{1/2}}$ is identified with 
the reduced right GNS space $\overline{\varphi^{1/2}A[\varphi]}$ by 
$J_\varphi(a\varphi^{1/2}) = (a\varphi^{1/2})^\natural = \varphi^{1/2} a^*$ for $a \in [\varphi]A[\varphi]$. 
Here $[\varphi]$ denotes the support projection of $\varphi$, 
$J_\varphi$ stands for the modular conjugation associated to the GNS vector $\varphi^{1/2}$ of 
the reduced algebra $[\varphi]A[\varphi]$ and the canonical *-operation, 
which is designated by $\natural$ here, is well-defined 
on the whole $L^2(A)$. Note that $\varphi$ is faithful when restricted to 
$[\varphi]A[\varphi]$. See \cite{yama1} and \cite{yama5} for further information. 

For a W*-bimodule ${}_AX_B$, we write
\[ 
A(-1/2)X = \Hom({}_AL^2(A),{}_AX)^\circ, 
\quad  
XB(-1/2) = \Hom(L^2(B)_B,X_B)
\] 
with the obvious operations of these on $X$ by right and left multiplications respectively,
which are $A$-$B$ bimodules by 
$\alpha(afb) = ((\alpha a)f)b$ and $(agb)\beta = a(g(b\beta))$ for 
$f \in A(-1/2)X$, $g \in XB(-1/2)$, $a \in A$, $b \in B$, $\alpha \in L^2(A)$ and $\beta \in L^2(B)$. 
(The circle for opposite algebra is placed in the definition of $A(-1/2)X$ 
to indicates that it acts on $X$ from the right.) 
Moreover, with this convention, we introduce an $A$-valued inner product ${}_A[\ ,\ ]$ on $A(-1/2)X$ and 
a $B$-valued inner product $[\ ,\ ]_B$ on $XB(-1/2)$ by 
\[ 
\alpha ({}_A[f',f]) = (\alpha f')f^*, 
\quad 
([g,g']_B)\beta = g^*(g'\beta). 
\]
Here $f,f' \in A(-1/2)X$ and $g,g' \in XB(-1/2)$. 
Note that $L^2(B)B(-1/2) = \End(L^2(B)_B)$ is equal to $B$, 
whereas $A(-1/2) L^2(A) = \End({}_AL^2(A))^\circ$ is indetified with $A$ by the right action of $A$. 
Notice that ${}_A[a'f',af] = a'{}_A[f',f]a^*$ and $[gb,gb']_B = b^* [g,g']_B b'$ 
for $a,a' \in A$ and $b, b' \in B$. 

Given index sets $I,J$, we introduce a matrix extension of a W*-algebra $N$ by 
$M_{I,J}(N) = \Hom(\ell^2(J)\otimes L^2(N)_N, \ell^2(I)\otimes L^2(N)_N)$, which is identified with 
a subspace of bounded $N$-valued matricial sequences in $N^{I\times J}$ and each $(x_{i,j}) \in M_{I,J}(N)$ 
is approximated by its finitely supported cuts in any weaker operator topology. 
Note that $M_I(N) = M_{I,I}(N)$ is a von Neumann algebra on $\ell^2(I)\otimes L^2(N)$ and 
weaker operator topologies are well-defined on 
$M_{I,J}(N)$ as a corner subspace of $M_{I\sqcup J}(N)$. 

The $L^2$-version of matrix extension is introduced analogously: 
A matrix extension (a Hilbert-Schmidt extension) of a W*-bimodule ${}_AX_B$ is defined by 
\[ 
{}^IX^J = \{ (\xi_{i,j}); \xi_{i,j} \in X, \sum_{i \in I, j \in J} \| \xi_{i,j}\|^2 < \infty\},  
\] 
which is an ${}_{M_I(A)}\text{W*}_{M_J(B)}$ bimodule in an obvious way. 

Recall that unilateral W*-modules are projective in the sense that 
we can find an index set $I$ and a projection $e \in M_I(A)$ so that 
${}_AX \cong {}_AL^2(A)^{I}\,e$ or $X_A \cong e\,{}^IL^2(A)_A$. 
This is nothing but a paraphrase of the Dixmier's structure theorem on normal 
*-homomorphisms between von Neumann algebras. 

\section{W*-bicategories of W*-bimodules}


As observed in \cite{yama2} from the view point of modular algebras, 
W*-bimodules constitute an involutive W*-bicategory, which we shall reconstruct here based on  
operator-valued inner products. 
According to two possibilities of them, there are two ways of forming tensor product bimodules, 
which are discriminatingly denoted by 
\[ 
X\leftthreetimes Y = \overline{XB(-1/2)\otimes_B Y}, 
\quad 
X\rightthreetimes Y = \overline{X\otimes_B B(-1/2)Y}. 
\]
Here algebraic module tensor products 
$XB(-1/2)\otimes_BY$ and $X\otimes_B B(-1/2)Y$ are pre-Hilbert spaces 
with their inner products defined by 
\[ 
(x\otimes_B\eta|x'\otimes_B\eta') = (\eta|[x,x']_B \eta'), 
\quad 
(\xi\otimes_By|\xi'\otimes_B y') = (\xi|\xi' {}_B[y',y])
\]
respectively and the bar denotes the Hilbert space completion. 

Given morphisms $f:{}_AX_B \to {}_AX'_B$ and $g:{}_BY_C \to {}_BY'_C$, bounded $A$-$C$ linear maps 
$f\leftthreetimes g: X\leftthreetimes Y \to X'\leftthreetimes Y'$ and 
$f\rightthreetimes g: X\rightthreetimes Y \to X'\rightthreetimes Y'$ are well-defined by 
\[ 
(f\leftthreetimes g)(x\otimes_B \eta) = (fx)\otimes_B (g\eta), 
\quad 
(f\rightthreetimes g)(\xi \otimes_B y) = (f\xi)\otimes_B (gy).  
\] 
Here $gy \in B(-1/2)Y$ is defined by $\beta (gy) = g(\beta y)$ ($\beta \in L^2(B)$). 

Unit isomorphisms are defined by 
\begin{align*} 
l^\leftthreetimes_{A,B}: L^2(A)\leftthreetimes X 
&= \overline{A\otimes_AX} \ni a\otimes_A\xi \mapsto a\xi \in X,\\
r^\leftthreetimes_{A,B}: X \leftthreetimes L^2(B)
&= \overline{XB(-1/2)\otimes_B L^2(B)} \ni x\otimes_B \beta \mapsto x\beta \in X 
\end{align*}
and 
\begin{align*} 
l^\rightthreetimes_{A,B}: L^2(A)\rightthreetimes X 
&= \overline{L^2(A)\otimes_A A(-1/2)X} \ni \alpha \otimes_A x \mapsto \alpha x \in X,\\
r^\rightthreetimes_{A,B}: X \rightthreetimes L^2(B)
&= \overline{X\otimes_B B} \ni \xi\otimes_B b \mapsto \xi b \in X.  
\end{align*}

To introduce associativity isomorphisms, we remark that the algebraic module tensor product 
$XB(-1/2)\otimes_B YC(-1/2)$ is canonically embedded into $(X\leftthreetimes Y)C(-1/2)$ 
in such a way that 
$(XB(-1/2)\otimes_B YC(-1/2))\otimes_C Z$ is dense in 
$(X\leftthreetimes Y)C(-1/2)\otimes_C Z$. 
Likewise, we have a canonical embedding $B(-1/2)Y\otimes_C C(-1/2)Z \subset 
B(-1/2)(Y\rightthreetimes Z)$ so that 
$X\otimes_B B(-1/2)Y\otimes_C C(-1/2)Z$ is dense in $X\rightthreetimes (Y\rightthreetimes Z)$. 
Associativity isomorphisms are now defined by 
\[ 
a^\leftthreetimes_{X,Y,Z}: 
(X\leftthreetimes Y)\leftthreetimes Z \ni (x\otimes_B y)\otimes_C \zeta 
\mapsto x \otimes_B (y\otimes_C \zeta) \in X\leftthreetimes (Y\leftthreetimes Z) 
\]
for $x \in XB(-1/2)$, $y \in YC(-1/2)$ and $\zeta \in Z$. 
\[ 
a^\rightthreetimes_{X,Y,Z}: 
(X\rightthreetimes Y)\rightthreetimes Z \ni (\xi\otimes_B y)\otimes_C z 
\mapsto \xi \otimes_B (y\otimes_C z) \in X\rightthreetimes (Y\rightthreetimes Z) 
\]
for $\xi \in X$, $y \in B(-1/2)Y$ and $z \in C(-1/2)Z$. 

The pentagon identity on a quadruple product $W\leftthreetimes X \leftthreetimes Y \leftthreetimes Z$ 
then follows from that on $WA(-1/2)\otimes_A XB(-1/2)\otimes_B YC(-1/2)\otimes_C Z$. 
Similarly for the $\rightthreetimes$ product. 

The triangle identities for unit isomorphisms are also witnessed on dense subspaces. 
For $X\leftthreetimes L^2(B) \leftthreetimes Y$, this is reduced to the commutativity of 
the diagram  
\[
\begin{CD}
XB(-1/2)\otimes_B B \otimes_B Y @>>> X\leftthreetimes (L^2(B) \leftthreetimes Y)\\
@VVV @VVV\\
(X\leftthreetimes L^2(B)) \leftthreetimes Y @>>> X\leftthreetimes Y
\end{CD}\ , 
\]
which is traced by 
\[ 
\begin{CD}
x\otimes_B b\otimes_B \eta @>>> x\otimes_B (b\otimes_B \eta)\\
@VVV @VVV\\
(x\otimes_B b)\otimes_B \eta @>>> xb\otimes_B \eta = x\otimes_B b\eta
\end{CD} 
\] 
for $x \in XB(-1/2)$, $b \in B$ and $\eta \in Y$. 

In this way, we have two W*-bicategories of W*-bimodules, which are denoted by 
$\cM^\leftthreetimes$ and $\cM^\rightthreetimes$ from here on.

\begin{Remark}
In \cite{take}, the associativity isomorphism is captured as 
$(X\leftthreetimes Y) \rightthreetimes Z \cong X\leftthreetimes (Y \rightthreetimes Z)$ 
in our notation.
Although $X \leftthreetimes Y \cong X \rightthreetimes Y$ in a canonical way 
(see below, cf.~\cite{conn}, \cite{sauv} also), the existence of these isomorphisms 
does not automatically mean the coherence for quadruple tensor products. 
\end{Remark}

\section{Canonical Equivalence}
Two W*-bicategories of W*-bimodules are now shown to be canonically equivalent. 
This is recognized in \cite{yama2} through natural identifications in modular tensor products. 
Here we shall establish this by constructing an explicit functor of equivalence. 

We first observe how tensor products behave under matrix extensions. Consider 
W*-bimodules ${}_AX_B$, ${}_BY_C$ and their column and row extensions 
${}^IX_B$ and ${}_BY^J$ by index sets $I,J$. 
Then $({}^IX) \leftthreetimes (Y^J)$ and $({}^IX) \rightthreetimes (Y^J)$ are naturally 
identified with Hilbert-Schmidt extensions 
${}^I(X\leftthreetimes Y)^J$ and ${}^I(X\rightthreetimes Y)^J$ 
of $X\leftthreetimes Y$ and $X\rightthreetimes Y$ respectively. 
Note that algebraic sums $\oplus_IXB(-1/2)$ and $B(-1/2)Y^{\oplus J}$ are weakly dense 
in $({}^IX)B(-1/2)$ and $B(-1/2)(Y^J)$ respectively. 

When this observation is applied to the standard bimodule $L^2(B)$, 
the unit isomorphisms $l^\leftthreetimes = r^\leftthreetimes:L^2(B)\leftthreetimes L^2(B) \to L^2(B)$ and 
$l^\rightthreetimes = r^\rightthreetimes: L^2(B)\rightthreetimes L^2(B) \to L^2(B)$ 
are enhanced to $M_I(B)$-$M_J(B)$ linear unitary maps 
\[
m_\odot: ({}^IL^2(B))\odot (L^2(B)^J) \to {}^I(L^2(B) \odot L^2(B))^J \to 
{}^I L^2(B)^J,
\]
where $\odot = \leftthreetimes$ or $\rightthreetimes$. 
We set 
\[ 
{}^Im^J = (m_\rightthreetimes)^* m_\leftthreetimes: 
({}^IL^2(B))\leftthreetimes (L^2(B)^J) \to 
({}^IL^2(B))\rightthreetimes (L^2(B)^J),  
\]
which is a unitary isomorphism in ${}_{M_I(B)}\cM_{M_J(B)}$. 


For ${}_AX_B$ and ${}_BY_C$, a unitary isomorphism 
$m_{X,Y}: X\leftthreetimes Y \to X\rightthreetimes Y$ is defined, with the help of 
projective-module realizations $u: X_B \to p\,{}^IL^2(B)_B$ and $v:{}_BY \to {}_BL^2(B)^J\,q$ as $B$-modules,  
by the commutativity of the diagram 
\[ 
\begin{CD}
X\leftthreetimes Y @>{u\leftthreetimes v}>> p\Bigl( ({}^I L^2(B))\leftthreetimes (L^2(B)^J) \Bigr)q\\
@V{m_{X,Y}}VV @VV{{}^Im^J}V\\
X\rightthreetimes Y @>>{u\rightthreetimes v}> p\Bigl( ({}^I L^2(B))\rightthreetimes (L^2(B)^J) \Bigr)q
\end{CD}.   
\]
Note here that 
$(p\,{}^IL^2(B)) \odot (L^2(B)^J\, q) =  p\Bigl( ({}^I L^2(B))\odot (L^2(B)^J) \Bigr)q$ 
for $\odot = \leftthreetimes$ or $\rightthreetimes$. 

By the bimodule linearity of ${}^Im^J$, 
$m_{X,Y}$ is $A$-$C$ linear and independent of the choice of projective-module realizations. 
Furthermore, $m_{X,Y}$ is natural in $X$ and $Y$ as well: 
For $f \in \Hom(X,X')$ and $g \in \Hom(Y,Y')$, the diagram 
\[ 
\begin{CD}
X \leftthreetimes Y @>{f\leftthreetimes g}>> X'\leftthreetimes Y'\\
@V{m_{X,Y}}VV @VV{m_{X',Y'}}V\\
X \rightthreetimes Y @>{f\rightthreetimes g}>> X'\rightthreetimes Y'
\end{CD}
\] 
is commutative. 

The following is immediate from the definition of $m_{X,Y}$. 

\begin{Lemma}\label{triangles}
Let $X = {}_AX_B$ be a W*-bimodule. Then the following diagrams commute. 
\[
\begin{xy}
(0,0) *{L^2(A)\leftthreetimes X}="A",
(40,0)*{L^2(A)\rightthreetimes X}="B",
(20,-20)*{X\,}="C",
\ar "A";"B"^{m_{L^2(A),X}}
\ar "A";"C"_{l_X^\leftthreetimes}
\ar "B";"C"^{l_X^\rightthreetimes}
\end{xy}, 
\hspace{10pt}
\begin{xy}
(0,0) *{X\leftthreetimes L^2(B)}="A",
(40,0)*{X\rightthreetimes L^2(B)}="B",
(20,-20)*{X\,}="C",
\ar "A";"B"^{m_{X,L^2(B)}}
\ar "A";"C"_{r_X^\leftthreetimes}
\ar "B";"C"^{r_X^\rightthreetimes}
\end{xy}. 
\]
\end{Lemma}


\begin{Theorem}
The identity functor gives a monoidal equivalence between $\cM^\leftthreetimes$ and 
$\cM^\rightthreetimes$ with respect to the multiplicativity isomorphisms $\{ m_{X,Y}\}$, i.e., 
the diagram 
\[ 
\begin{CD}
(X\leftthreetimes Y) \leftthreetimes Z @>{m_{X,Y}\leftthreetimes 1}>> 
(X\rightthreetimes Y) \leftthreetimes Z
@>{m_{X\leftthreetimes Y,Z}}>> (X\rightthreetimes Y) \rightthreetimes Z\\ 
@V{a^\leftthreetimes}VV @. @VV{a^\rightthreetimes}V\\
X\leftthreetimes (Y \leftthreetimes Z) @>>{1\leftthreetimes m_{Y,Z}}> 
X\leftthreetimes (Y \rightthreetimes Z)
@>>{m_{X,Y\rightthreetimes Z}}> X\rightthreetimes (Y \rightthreetimes Z)
\end{CD}
\] 
is commutative for any composable triplets $X,Y,Z$ of W*-bimodules. 
\end{Theorem}

\begin{proof} 
By projective-module realizations of $X$ and $Z$ together with the naturality of relevant morphisms, 
the problem is reduced to the case $X = L^2(B)$ and $Z = L^2(C)$ with $Y = {}_BY_C$, 
whose validity can be seen 
from the following division of the diagram 
\[
\begin{xy}
(20,0) *{(I_B\leftthreetimes Y)\leftthreetimes I_C}="A",
(20,-50)*{I_B\leftthreetimes (Y\leftthreetimes I_C)}="B",
(70,-50)*{I_B\leftthreetimes (Y \rightthreetimes I_C)}="C",
(120,-50)*{I_B\rightthreetimes (Y \rightthreetimes I_C),}="D",
(70,0)*{(I_B\rightthreetimes Y) \leftthreetimes I_C}="E",
(120,0)*{(I_B\rightthreetimes Y) \rightthreetimes I_C}="F",
(45,-15)*{I_B\leftthreetimes Y}="G",
(95,-15)*{I_B \rightthreetimes Y}="H",
(95,-35)*{Y\rightthreetimes I_C}="I",
(45,-35)*{Y\leftthreetimes I_C}="J",
(70,-25)*{Y}="K",
(30,-20)*{\textcircled{\scriptsize1}}, 
(45,-25)*{\textcircled{\scriptsize2}},
(60,-8)*{\textcircled{\scriptsize2}},
(69,-20)*{\textcircled{\scriptsize3}},
(95,-8)*{\textcircled{\scriptsize3}},
(110,-20)*{\textcircled{\scriptsize1}},
(92,-25)*{\textcircled{\scriptsize2}},
(60,-42)*{\textcircled{\scriptsize2}},
(70,-31)*{\textcircled{\scriptsize3}},
(94,-44)*{\textcircled{\scriptsize3}},
\ar "A";"E"
\ar "F";"D"
\ar "A";"B"
\ar "C";"D"
\ar "B";"C"
\ar "E";"F"
\ar "G";"H"
\ar "A";"G"
\ar "B";"J"
\ar "E";"H" 
\ar "C";"I" 
\ar "F";"H"
\ar "D";"I"
\ar "B";"G" 
\ar "J";"K"
\ar "J";"I"
\ar "I";"K"
\ar "G";"K"
\ar "D";"H" 
\ar "H";"K"
\end{xy}
\]
where diagrams around $\textcircled{\scriptsize1}$ commute by 
the triangle identity for unit isomorphisms, 
diagrams around $\textcircled{\scriptsize2}$ commute by the naturality of unit isomorphisms 
and diagrams around $\textcircled{\scriptsize3}$ commute by Lemma~\ref{triangles}.  
\end{proof}



\section{Unitary Involutions}
Given a W*-bimodule ${}_AX_B$, the dual Hilbert space $X^*$ is naturally 
a ${}_B\text{W*}_A$ bimodule so that the operation of taking duals gives a contravariant functor 
${}_A\cM_B \to {}_B\cM_A$ with the operation on morphisms 
given by taking the transposed ${}^tf: Y^* \to X^*$ of $f \in \Hom(X,Y)$. 
With the notation $\xi^*$ ($\xi \in X$) to stand for a linear form 
$X \ni \xi' \mapsto (\xi|\xi')$ (the inner product being linear in the second variable), 
${}^tf$ is described by 
$\langle {}^tf\eta^*,\xi\rangle = (\eta|f\xi)$. 
The operation is then involutive (so-called self-duality on Hilbert spaces) 
in the sense that, if we denote the canonical isomorphism $(X^*)^* \cong X$ 
by $d_X: X \ni \xi \mapsto \xi^{**} = (\xi^*)^* \in X^{**}$, 
it is natural in $X$, satisfies ${}^td_X = d_{X^*}^{-1}$ and 
gives an equivalence between the iterated involution and the identity functor. 
 
As to the monoidal structures in $\cM$, the dualizing functor gives an anti-multiplicative equivalence 
between $\cM^\leftthreetimes$ and $\cM^\rightthreetimes$. 

To see this, we begin with some preparatory discussions. 
For $x \in XB(-1/2) = \Hom(L^2(B)_B,X_B)$ with 
$X = {}_AX_B$ a W*-bimodule in ${}_A\cM_B$, 
define its conjugate $\overline x \in \Hom({}_BL^2(B),{}_BX^*)$ by 
$\overline{x}(\beta) = (x(\beta^\natural))^*$ and set 
\[ 
x^\star = {\overline{x}}^\circ \in 
B(-1/2)X^* = \Hom({}_BL^2(B),{}_BX^*)^\circ. 
\] 
Recall that $\beta^\natural$ denotes the natural *-operation on $L^2(B)$.

Then the correspondence $x \mapsto x^\star$ gives 
a conjugate-linear isometric isomorphism between $X B(-1/2)$ and $B(-1/2) X^*$ in such a way that 
$(axb)^\star = b^*x^\star a^*$ and 
$[x',x]_B^* = {}_B[x^\star,(x')^\star]$ for $a \in A$, $b \in B$ and $x,x' \in XB(-1/2)$. 

We now introduce a natural (covariant) family of unitary morphisms 
$c_{X,Y}: Y^*\rightthreetimes X^* \to (X\leftthreetimes Y)^*$ in ${}_C\cM_A$ 
($Y = {}_BY_C$) by 
$c_{X,Y}(\eta^*\otimes_B x^\star) = (x\otimes_B \eta)^*$ ($x \in XB(-1/2)$, $\eta \in Y$), 
which is anti-multiplicative in the sense that the following hexagon diagram commutes. 
\[ 
\begin{CD}
(Z^*\rightthreetimes Y^*)\rightthreetimes X^* @>{c\otimes 1}>> 
(Y\leftthreetimes Z)^*\rightthreetimes X^* @>{c}>> 
(X\leftthreetimes(Y\leftthreetimes Z))^*\\
@V{a}VV & & @VV{{}^ta}V\\
Z^*\rightthreetimes(Y^*\rightthreetimes X^*) @>>{1\otimes c}> 
Z^*\rightthreetimes(X\leftthreetimes Y)^* @>>{c}> 
((X\leftthreetimes Y)\leftthreetimes Z)^*  
\end{CD}\ .
\]
To see this, let $x \in XB(-1/2)$, $y \in YC(-1/2)$ and $\zeta \in Z$. 
Then the above diagram is traced by  
\[ 
\begin{CD}
(\zeta^*\otimes_C y^\star)\otimes_B x^\star @>{c\otimes 1}>> 
(y\otimes_C \zeta)^*\otimes_B x^\star @>{c}>> 
(x\otimes_B(y\otimes_C \zeta))^*\\
@V{a}VV & & @VV{{}^ta}V\\
\zeta^*\otimes_C(y^\star \otimes_B x^\star) @>>> 
\zeta^*\otimes_C(x\otimes_B y)^\star @>>{c}> 
((x\otimes_B y)\otimes_C \zeta)^*  
\end{CD}
\]
and the hexagonal commutativity is reduced to the equality $(x\otimes_By)^\star = y^\star\otimes_Bx^\star$ 
in $((X\leftthreetimes Y)C(-1/2))^\star$, which is in turn checked by 
\[ 
\gamma (x\otimes y)^\star = ((x\otimes y)(\gamma^*))^* \mapsto 
y(\gamma^*)^*\otimes x^\star = \gamma (y^\star\otimes x^\star)
\quad 
\text{for $\gamma \in L^2(C)$.}
\] 

Being prepared, we define anti-multiplicativity isomorphisms 
$c_{X,Y}^\rightthreetimes: Y^*\rightthreetimes X^* \to (X\rightthreetimes Y)^*$ in $\cM^\rightthreetimes$ to be 
the composition $c_{X,Y}^\rightthreetimes = {}^tm_{X,Y}^{-1} c_{X,Y}$, which together with the duality 
isomorphisms $\{ d_X\}$ constitute a unitary involution on $\cM^\rightthreetimes$: 
As a composition of anti-multiplicative $c_{X,Y}$ and multiplicative ${}^tm_{X,Y}^{-1}$, 
$c_{X,Y}^\rightthreetimes$ is anti-multicative, i.e., the hexagon identity holds. 

For the commutativity of 
\[
\begin{CD}
X\rightthreetimes Y @>{d\rightthreetimes d}>> X^{**}\rightthreetimes Y^{**}\\
@V{d}VV @VV{c^\rightthreetimes}V\\
(X\rightthreetimes Y)^{**} @>>{{}^tc^\rightthreetimes}> (Y^*\rightthreetimes X^*)^*
\end{CD}\ ,
\]
we first describe $c^\rightthreetimes$ in terms of standard spaces. 
Given a projective-module realization $u: p{}^IL^2(B)_B \cong X_B$, 
its transposed map is composed with the row-vector extension of the canonical isomorphism 
$L^2(B)^* \cong L^2(B)$ to get the accompanied isomorphim $v:{}_BX^* \cong {}_BL^2(B)^Ip$. 
These are then combined with $c_{X,Y}$ and ${}^tm_{X,Y}$ to form a commutative diagram 
\[
\begin{CD}
Y^* \rightthreetimes X^* @>{c_{X,Y}}>> (X\leftthreetimes Y)^* @<{{}^tm_{X,Y}}<< (X\rightthreetimes Y)^*\\
@VVV @VVV @VVV\\
Y^* \rightthreetimes L^2(B)^Ip @>>> (p {}^IL^2(B) \leftthreetimes Y)^* @<<<
(p {}^I(L^2(B) \rightthreetimes Y))^*
\end{CD}\ ,
\]
where the bottom line is described by the correspondences
\[ 
\eta^*\otimes_B (b_i^*)p \longleftrightarrow (p(b_i)\otimes_B\eta)^* 
\longleftrightarrow (p(\beta_i\otimes_B y))^*
\]
with the relation $p(b_i\eta) = p(\beta_i y)$ in ${}^IY$ assumed at the second one. 
Thus, replacing $\eta^*\otimes_B(b_i^*)p$ with $(\sum_i y^\star \beta_i^\natural\otimes_B 1_i)p$, 
$c_{X,Y}^\rightthreetimes$ is specified by the commutativity of the diagram
\[
\begin{CD}
Y^* \rightthreetimes X^* @>{c_{X,Y}^\rightthreetimes}>> (X\rightthreetimes Y)^*\\
@V{1\rightthreetimes v}VV @VV{{}^t(u\rightthreetimes 1)}V\\
Y^* \rightthreetimes L^2(B)^Ip @>>> (p{}^IL^2(B)\rightthreetimes Y)^*
\end{CD} 
\] 
with the bottom line given by 
$\left( \sum_i y^\star \beta_i^\natural\otimes 1_i\right)p \mapsto
(p(\beta_i) \otimes_By)^*$. Here $y \in B(-1/2)Y$, $\beta_i \in L^2(B)$ and 
$1_i = \delta_i \in B^I = B(-1/2) L^2(B)^I$ denotes the canonical row basis. 

By symmetry, $c_{X,Y}^\rightthreetimes$ is also described in terms of 
a projective-module realization ${}_BY \cong {}_BL^2(B)^Jq$ by the correspondence 
\[ 
q{}^J L^2(B) \rightthreetimes X^* \ni 
q(\beta_j^\natural)\otimes_B x^\star \mapsto 
(\sum_j (x\beta_j\otimes 1_j)q)^* 
\in ((X\rightthreetimes L^2(B))^Jq)^*. 
\] 

Now the square identity for duality isomorphisms takes the form 
\[ 
\begin{CD}
p{}^IL^2(B) \rightthreetimes Y @>{1\rightthreetimes d}>> p {}^IL^2(B) \rightthreetimes Y^{**}\\
@V{d}VV @VV{c^\rightthreetimes}V\\
(p {}^I L^2(B) \rightthreetimes Y)^{**} @>>{{}^tc^\rightthreetimes}> (Y^* \rightthreetimes L^2(B)^Ip)^*
\end{CD}\ . 
\] 
Here canonical isomorphisms 
\[ 
(p{}^IL^2(B))^{**} \cong p {}^IL^2(B), 
\quad 
(p {}^IL^2(B))^* \cong L^2(B)^Ip
\]
are used at the right corners with $c^\rightthreetimes$ and ${}^tc^\rightthreetimes$ modified 
accordingly. 
The diagram is then traced by   
\[ 
\begin{CD}
p(\beta_i)\otimes_By @>>> p(\beta_i)\otimes_B y^{\star\star}\\
@VVV @VVV\\
(p(\beta_i)\otimes_B y)^{**} @>>> ((y^\star \beta_i^\natural \otimes_B 1)p)^*
\end{CD}
\]
and the commutativity holds. 

In this way, we have checked that $\{ c_{X,Y}^\rightthreetimes \}$ defines a unitary involution on 
$\cM^\rightthreetimes$. 
Likewise $c_{X,Y}^\leftthreetimes = c_{X,Y} m_{Y^*,X^*}$ gives a unitary involution on $\cM^\leftthreetimes$ 
so that $\{ m_{X,Y}\}$ intertwines these. 
As a conclusion, we have 

\begin{Theorem}
Anti-multiplicativity isomorphisms 
$\{ c_{X,Y}^\leftthreetimes\}$ and $\{ c_{X,Y}^\rightthreetimes\}$ define unitary involutions on
$\cM^\leftthreetimes$ and $\cM^\rightthreetimes$ respectively so that they are equivalent through 
the monoidal equivalence $\{ m_{X,Y}\}$. 
\end{Theorem} 



\end{document}